\documentclass[12pt,letterpaper,reqno]{amsart}

\usepackage{amsmath, amssymb, latexsym, amsthm}
\usepackage{fullpage}
\usepackage{graphicx}
\usepackage{enumerate}
\usepackage[flushmargin]{footmisc}
\usepackage{verbatim}
\usepackage{scrextend}
\usepackage{color} 

\newtheorem{thm}{Theorem}
\newtheorem{lem}[thm]{Lemma}
\newtheorem{cor}[thm]{Corollary}
\newtheorem{defn}{Definition}
\newtheorem{claim}[thm]{Claim}
\newtheorem{problem}{Problem}

\usepackage{tikz}
\usetikzlibrary{shapes,decorations.pathreplacing,calc}
\newcommand{\myvert}{ellipse (.1cm and .1cm)}

\newcommand{\wtilpi}{\widetilde{\pi}}
\newcommand{\s}[1]{^{(#1)}}

\DeclareMathOperator{\ex}{ex}
\DeclareMathOperator{\dist}{dist}

\newcommand{\CL}[1]{\left\lceil#1\right\rceil}
\newcommand{\fl}[1]{\lfloor#1\rfloor}
\newcommand{\FL}[1]{\left\lfloor#1\right\rfloor}

\DeclareFontFamily{U}{mathx}{\hyphenchar\font45}
\DeclareFontShape{U}{mathx}{m}{n}{
      <5> <6> <7> <8> <9> <10>
      <10.95> <12> <14.4> <17.28> <20.74> <24.88>
      mathx10
      }{}
\DeclareSymbolFont{mathx}{U}{mathx}{m}{n}
\DeclareFontSubstitution{U}{mathx}{m}{n}
\DeclareMathAccent{\widecheck}{0}{mathx}{"71}
\DeclareMathAccent{\wideparen}{0}{mathx}{"75}

\author[Erbes, Ferrara, Martin, Wenger]{Catherine Erbes$^{1,5}$\and Michael Ferrara$^{2,6}$\and Ryan R. Martin$^{3,6,7}$ \and Paul Wenger$^4$}
\title{Stability of the potential function}
\date{\today}

\begin{document}
\footnotetext[1]{Department of Mathematics, Hiram College, {\tt erbescc@hiram.edu}}
\footnotetext[2]{Department of Mathematical and Statistical Sciences, University of Colorado Denver, {\tt michael.ferrara@ucdenver.edu}}
\footnotetext[3]{Department of Mathematics, Iowa State University, {\tt rymartin@iastate.edu}}
\footnotetext[4]{School of Mathematical Sciences, Rochester Institute of Technology, {\tt pswsma@rit.edu}}
\footnotetext[5]{Research supported in part by UCD GK12 Transforming Experiences Project, National Science Foundation Grant DGE-0742434.}
\footnotetext[6]{Research supported in part by Simons Foundation Collaboration Grants \#206692 (to M. Ferrara) and \#353292 (to R. R. Martin).}
\footnotetext[7]{Research supported in part by National Science Foundation grant DMS-0901008 and National Security Agency grant H98230-13-1-0226.}

\topmargin 0.4in

\begin{abstract}
A graphic sequence $\pi$ is \textit{potentially $H$-graphic} if there is some realization of $\pi$ that contains $H$ as a subgraph.  The Erd\H os-Jacobson-Lehel problem asks to determine $\sigma(H,n)$, the minimum even integer such that any $n$-term graphic sequence $\pi$ with sum at least $\sigma(H,n)$ is potentially $H$-graphic.  The parameter $\sigma(H,n)$ is known as the \textit{potential function} of $H$, and can be viewed as a degree sequence variant of the classical extremal function $\ex(n,H)$.  Recently, Ferrara, LeSaulnier, Moffatt and Wenger [On the sum necessary to ensure that a degree sequence is potentially $H$-graphic, \textit{Combinatorica} \textbf{36} (2016), 687--702] determined $\sigma(H,n)$ asymptotically for all $H$, which is analogous to the Erd\H{o}s-Stone-Simonovits Theorem that determines $\ex(n,H)$ asymptotically for nonbipartite $H$.  

In this paper, we investigate a stability concept for the potential number, inspired by Simonovits' classical result on the stability of the extremal function.  We first define a notion of stability for the potential number that is a natural analogue to the stability given by Simonovits.  However, under this definition, many families of graphs are not $\sigma$-stable, establishing a stark contrast between the extremal and potential functions.  We then give a sufficient condition for a graph $H$ to be stable with respect to the potential function, and characterize the stability of those graphs $H$ that contain an induced subgraph of order $\alpha(H)+1$ with exactly one edge.  
\end{abstract}

\maketitle

\section{Introduction}

The \textit{degree sequence} of a graph $G$ is the list of degrees of the vertices in $G$,  and a finite sequence of nonnegative integers $\pi=(d_1,\ldots,d_n)$ is {\it graphic} if it is the degree sequence of some graph $G$. In this case, we call $G$ a {\it realization} of $\pi$ or say that $G$ {\it realizes} $\pi$. Throughout this paper, we assume that all graphic sequences are nonincreasing.  Classical theorems, most notably the Havel-Hakimi algorithm \cite{hak, hav} and the Erd\H{o}s-Gallai criteria \cite{ErdGal}, give efficient characterizations of graphic sequences.  However, a given graphic sequence may have a number of nonisomorphic realizations that exhibit a breadth of properties. The general problem of exploring the properties appearing throughout the space of realizations of a graphic sequence is at the heart of this paper, and has been central to a number of investigations throughout the literature.  

Given a graph property ${\mathcal F}$, we say that a graphic sequence $\pi$ is \textit{potentially ${\mathcal F}$-graphic} if at least one realization of $\pi$ has property ${\mathcal F}$. It is possible to develop analogues to many traditional graph theoretic results and conjectures within the setting of potentially ${\mathcal F}$-graphic sequences.  Potentially-$\mathcal{F}$-graphic variants of Hadwiger's Conjecture \cite{Chen,DvoMoh}, the Graph Minor Theorem \cite{ChuSey,SRao2}, the Erd\H os-S\'os Conjecture \cite{YinLi2}, the Bollob\'as-Scott Conjecture on bisections \cite{HarSea}, and graph packing~\cite{BFHJKW, DFJS} have all recently been studied. 

Let $H$ be a graph.  In this paper, we study a graphic sequence variant of the classical Tur\'an problem, which asks for the maximum number of edges in an $n$-vertex graph that does not contain $H$ as a subgraph; this number is called the {\it extremal number} of $H$ and is denoted $\ex(n,H)$.  A graphic sequence $\pi$ is {\it potentially $H$-graphic} if there is a realization of $\pi$ that contains $H$ as a subgraph.  In~\cite{EJL}, Erd\H os, Jacobson, and Lehel posed the following general problem.
Let $\sigma(\pi)$ denote the sum of the terms of the degree sequence $\pi$.
\begin{problem}[The Erd\H os-Jacobson-Lehel Problem] Determine the minimum even integer $\sigma(H,n)$ such that any graphic sequence $\pi$ of length $n$ with $\sigma(\pi) \geq \sigma(H,n)$ is potentially $H$-graphic. \end{problem}
\noindent We call $\sigma(H,n)$ the {\it potential number} or the {\it potential function} of $H$.

The focus of this paper is the presentation of stability results for the potential function that parallel classical stability theorems for the extremal function.
In Section~\ref{sec:background} we outline motivating theorems from the stability literature on the Tur\' an problem.
In Section~\ref{sec:main} we precisely define the notion of stability for the Erd\H os-Jacobson-Lehel problem and state our main results.
In Section~\ref{sec:tech} we provide some technical lemmas and in Section~\ref{sec:proofs} we prove our main results.
In Section~\ref{sec:notstable}, we introduce a slightly weaker notion of stability that captures still more graphs, and we conclude with a discussion of possible future work.

Throughout the paper we use standard graph theoretic notation.
For a graph $G$ we let $V(G)$ and $E(G)$ denote its vertex set and edge set, respectively.
The maximum size of an independent set in $G$ is denoted by $\alpha(G)$, and  the maximum degree of $G$ is denoted by $\Delta(G)$.
If $X\subseteq V(G)$, then $G[X]$ is the subgraph of $G$ induced by $X$.
If $G$ and $H$ are graphs, then the {\it join} of $G$ and $H$, denoted $G\vee H$, is the graph obtained from the disjoint union of $G$ and $H$ by adding all possible edges joining vertices in $G$ to vertices in $H$.

\section{Background}\label{sec:background}
Since the Tur\'an problem and the related stability result proved by Erd\H{o}s and Simonovits are a major motivation for this work, we discuss them briefly. 
In~\cite{Turan}, Tur\'an determined $\ex(n,K_r)$ for all $r$, extending an earlier result of Mantel and his students for $r=3$~\cite{Mantel}.
Furthermore, Tur\'an proved that the unique $n$-vertex $K_r$-free graph with $\ex(n,H)$ edges is the Tur\'an graph $T_{n,r}$, which is the complete $r$-partite graph with all partite sets having order $\CL{n/r}$ or $\FL{n/r}$. 
Later Erd\H os and Simonovits \cite{ESS} extended work of Erd\H os and Stone \cite{ES} to determine the extremal number for general $H$ asymptotically.
\begin{thm}[The Erd\H{o}s-Stone-Simonovits Theorem]
If $H$ is a graph with chromatic number $\chi(H)=r+1\ge 2$, then $$\ex(n,H)=|E(T_{n,r})| + o(n^2).$$
\end{thm}

Subsequently, Erd\H{o}s \cite{Erdstab} and Simonovits \cite{Sim} independently proved the following result, which is sometimes referred to as the First Stability Theorem.
Given graphs $G$ and $G'$ on the same labeled vertex set, the {\it edit distance} between $G$ and $G'$, denoted $\dist(G,G')$, is $|E(G) \triangle E(G')|$, where $\triangle$ denotes the symmetric difference. 

\begin{thm}[Erd\H{o}s \cite{Erdstab}, Simonovits \cite{Sim}]\label{thm:stab2}
Let $H$ be a graph with $\chi(H)=r+1$. For every $\epsilon >0$, there exists a $\delta > 0$ and an $n_{\epsilon}$ such that if $n > n_{\epsilon}$ and $G$ is an $n$-vertex $H$-free graph such that 
\[
|E(G)| \geq \ex(n,H)-\delta n^2,
\]
then $\dist(G, T_{n,r}) < \epsilon n^2$.
\end{thm}

Intuitively, Theorem \ref{thm:stab2} states that as the number of edges in an $H$-free graph increases, the graph begins to converge to the appropriate Tur\'an graph under the metric of edit distance. Simonovits then used this result to determine the extremal number for $pK_r$, the graph consisting of $p$ disjoint copies of $K_r$~\cite{Sim}.  Simonovits' work gave rise to what has been dubbed the \textit{stability method}: combining an asymptotic solution to a given extremal problem, which yields a set $\mathcal{C}$ of extremal objects, and a stability result, which shows that objects that are close to being extremal must ``look like'' those in $\mathcal{C}$, we can show that all extremal objects are in fact in $\mathcal{C}$. Stability methods have been used to attack a wide variety of extremal problems (c.f. \cite{BBSW, Kee, Mub2, Nik, Pik}), including recent stability approaches to problems in Ramsey Theory \cite{GRSS, GSS, NS}, and the hypergraph Tur\'an problem (c.f. \cite{BT}, \cite{Mub1}, or \cite{MubPik}).  

\section{Main Results}\label{sec:main}

The potential number has been calculated exactly for a variety of specific families of graphs, most notably complete graphs~\cite{LiSongLuo}.
Recently, Ferrara, LeSaulnier, Moffatt, and Wenger \cite{FLMW} determined $\sigma(H,n)$ asymptotically for all $H$, which is an Erd\H{o}s-Stone-Simonovits-type result for the potential function.  Much as Erd\H{o}s-Stone-Simonovits provides the class of target graphs for the First Stability Theorem, the result of Ferrara et al. informs our stability results for $\sigma(H,n)$.  As such, we will describe their result fully here.  

Let $H$ be a graph of order $k$ with at least one nontrivial component, and let $\alpha(H)$ be the independence number of $H$. The results in~\cite{FLMW} depend on a family of $k-\alpha(H)$ degree sequences of length $n$ that are not potentially $H$-graphic.

\begin{defn}
For $i \in \{\alpha(H)+1, \ldots, k\}$, define
\[
\nabla_i(H)=\min \{\Delta(F): F \text{ is an induced subgraph of }H\text{ and }|V(F)|= i\}.\]
For $n$ sufficiently large, define
$$\wtilpi_i(H,n) = ((n-1)^{k-i}, (k-i+\nabla_i(H)-1)^{n-k+i}),$$
where the exponents denote the multiplicities of the terms in the sequence. 
This is a graphic sequence as long as $n-k+i$ and $\nabla_i(H)-1$ are not both odd. In that case, we reduce the last term of this sequence by $1$.
\end{defn}

We claim that for all $i$, $\wtilpi_i(H,n)$ is not potentially $H$-graphic.
Every realization $G$ of $\wtilpi_i(H,n)$ consists of a clique of order $k-i$ that is joined to a graph with maximum degree $\nabla_i(H)-1$.
Every set of $k$ vertices in $G$ contains at least $i$ vertices that are not in the dominating clique, and these vertices cannot induce an $i$-vertex graph with maximum degree at least $\nabla_i$.
Thus $G$ cannot contain $H$ as a subgraph, implying that $$\sigma(H,n) \geq \max_i\{ \sigma(\wtilpi_i(H,n))\}+2.$$

As we are concerned with the asymptotics of the potential function, we need only consider the leading coefficient of $\sigma(\wtilpi_i(H,n))$. This coefficient is denoted $\widetilde\sigma_i(H)$.  A quick computation shows that $\widetilde\sigma_i(H)=2(k-i)+\nabla_i(H)-1$, so we define
\begin{align*}
\widetilde\sigma(H)&=\max_i\{\widetilde\sigma_i(H)\}\\
&=\max_i\{2(k-i)+\nabla_i(H)-1\}.
\end{align*}

With this notation in hand, we are ready to state the main result of \cite{FLMW}.
\begin{thm}[Ferrara, LeSaulnier, Moffatt, and Wenger \cite{FLMW}]\label{FLMW}
If $H$ is a graph and $n$ is a positive integer, then $$\sigma(H,n)=\widetilde\sigma(H)n + o(n).$$
\end{thm}

Note that this value is maximized when $2i-\nabla_i(H)$ is minimized. 

\begin{defn}
For a graph $H$ and $n$ sufficiently large, define
\begin{align*}
\mathcal{P}(H,n)&=\{ \wtilpi_i(H,n) \colon \widetilde\sigma_i(H)= \widetilde\sigma(H)\}\\
&=\{\wtilpi_i(H,n) \colon i\in\arg\min_i \{2i-\nabla_i(H)\}\}.
\end{align*}
\end{defn}

We note that there are graphs $H$ for which $\mathcal{P}(H,n)$ contains multiple sequences (examples include $H=K_k-P_3$ and $H=K_{k-5}\vee P_5$).

In order to define a stability concept for the potential function, we require a measure of distance between two graphic sequences. Given graphic sequences $\pi_1 = (x_1, \ldots, x_n)$ and $\pi_2 = (y_1, \ldots, y_m)$ with $m \le n$, we let $\| \pi_1-\pi_2\| =  \sum_{j=1}^n |x_j-y_j|$, where we define $y_{m+1}, \ldots, y_n$ to be 0 if $m \neq n$.
Note that this is the $\ell^1$ norm of $\pi_1-\pi_2$, 
which aligns with edit distance, since if $G$ is a graph for which $\dist(G, T_{n,r}) < \epsilon n$, then $\|\pi(G)-\pi(T_{n,r})\|< 4 \epsilon n$. 
 
\begin{defn} A graph $H$ is {\bf stable with respect to the potential number}, or {\bf $\sigma$-stable}, if for any $\epsilon> 0$, there exists a $\delta > 0$ and an $n_0=n(\epsilon, H)$ such that for any graphic sequence $\pi$ of length $n \geq n_0$ that is not potentially $H$-graphic and that satisfies
$$\sigma (\pi) \geq \sigma(H,n)-\delta n,$$
there is some $\pi' \in \mathcal{P}(H,n)$ such that $\|\pi-\pi'\| < \epsilon n$.
\end{defn}

To demonstrate the concept of $\sigma$-stability, we provide an example of a graph that is not $\sigma$-stable.
This draws an immediate contrast with the extremal function, where Theorem \ref{thm:stab2} demonstrates that every nonbipartite graph is stable with respect to the extremal function.  

Erd\H os, Jacobson, and Lehel proved that $\sigma(K_3,n)=2n$.
It is straightforward to check that $\mathcal P(K_3,n)$ consists of a single sequence, namely $(n-1,1^{n-1})$, which is uniquely realized by the star of order $n$.
However, the sequence $\pi=(\CL{\frac n2},\FL{\frac n2},1^{n-2})$, whose only realization is the (nearly) balanced double star, also has sum $2n-2$ and is not potentially $K_3$-graphic (see Figure~\ref{Fig:K3not}).
Thus $\sigma(\pi)=\sigma(K_3,n)-2$ and for all $\pi'\in \mathcal P(K_3,n)$ we have $||\pi-\pi'||\ge n-3$.
Therefore $K_3$ is not $\sigma$-stable, and a similar examination of $((n-1)^k,k^{n-1})$ and $\pi=((n-1)^{k-1}, \CL{\frac n2}+(k-1),\FL{\frac n2}+(k-1),k^{n-2})$ demonstrates that $K_k$ is not $\sigma$-stable for any $k\ge 3$.

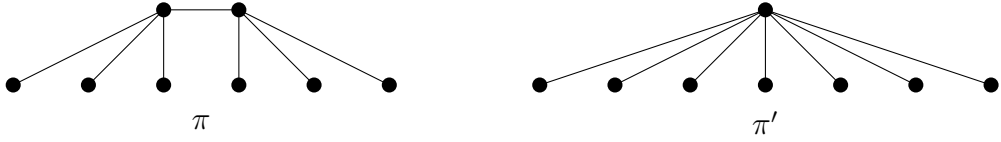
\begin{figure}
\begin{tikzpicture}

\fill (0,0) \myvert;
\fill (1,0) \myvert;
\fill (2,0) \myvert;
\fill (3,0) \myvert;
\fill (4,0) \myvert;
\fill (5,0) \myvert;
\fill (2,1) \myvert;
\fill (3,1) \myvert;
\draw (0,0)--(2,1);
\draw (1,0)--(2,1);
\draw (2,0)--(2,1);
\draw (2,1)--(3,1);
\draw (3,0)--(3,1);
\draw (4,0)--(3,1);
\draw (5,0)--(3,1);

\node at (2.5,-.5) {$\pi$};

\fill (7,0) \myvert;
\fill (8,0) \myvert;
\fill (9,0) \myvert;
\fill (10,0) \myvert;
\fill (11,0) \myvert;
\fill (12,0) \myvert;
\fill (13,0) \myvert;
\fill (10,1) \myvert;
\draw (7,0)--(10,1);
\draw (8,0)--(10,1);
\draw (9,0)--(10,1);
\draw (10,0)--(10,1);
\draw (11,0)--(10,1);
\draw (12,0)--(10,1);
\draw (13,0)--(10,1);

\node at (10,-.5) {$\pi'$};

\end{tikzpicture}
\caption{The unique realizations of $\pi$ and $\pi'$, two degree sequences on $8$ vertices that are not potentially $K_3$-graphic.
It follows that $K_3$ is \emph{not} $\sigma$-stable since $\mathcal P(K_3,n)=\{\pi'\}$ and $\|\pi-\pi'\|> n/3$.}\label{Fig:K3not}
\end{figure}

Let $i^*(H)$ be the smallest index $i \in \{\alpha(H)+1, \ldots, k\}$ for which $\wtilpi_{i}(H,n)\in \mathcal P(H,n)$.
Recall that this means that $i^*(H)$ is the smallest index $i$ for which $2i-\nabla_i$ is minimized. 
Our main theorem describes a large class of graphs that are $\sigma$-stable. 

\begin{thm}\label{thm:MainLow}
If $H$ is a graph such that  $2i^*(H)-\nabla_{i^*}(H) \leq 2\alpha(H)$, then $H$ is $\sigma$-stable. 
\end{thm}

Henceforth, when $H$ is understood, we will suppress the argument $H$ in our use of parameters like $\alpha, \nabla_i$, and $i^*$.  For all choices of $H$, because $\nabla_{\alpha +1}\ge 1$, we have \[ 2i^*-\nabla_{i^*} \leq 2(\alpha+1)-\nabla_{\alpha+1} \leq 2\alpha+1.\] 
Therefore, Theorem~\ref{thm:MainLow} applies to $H$ unless $2i^*-\nabla_{i^*} = 2\alpha +1$. 
In this case, $2(\alpha+1)-\nabla_{\alpha+1}\ge 2\alpha+1$, so we conclude that $\nabla_{\alpha+1}=1$.
Therefore there is a set of $\alpha+1$ vertices in $H$ that induces a graph consisting of a matching and isolated vertices.
It is worth noting that $2i^*-\nabla_{i^*} \leq 2\alpha$ does not necessarily imply that $\nabla_{\alpha+1}(H)>1$. 

Our next result applies to certain graphs for which $2i^*-\nabla_{i^*}=2\alpha+1$, namely those with an induced subgraph of order $\alpha+1$ that has exactly one edge.  Let $S_{x,y}$ be the double star with central vertices of degree $x+1$ and $y+1$.

\begin{thm}\label{thm:MainHigh}
If $H$ is a graph of order $k$ and independence number $\alpha$ such that 
\begin{itemize}
\item $2i^*-\nabla_{i^*}= 2\alpha+1$ and
\item $H$ has a set $X$ of $\alpha+1$ vertices such that $H[X]$ has exactly one edge,
\end{itemize}
then $H$ is $\sigma$-stable if and only if $H\subseteq K_{k-\alpha(H)-2} \vee S_{b_1, b_2}$ for some $b_1$ and $b_2$ with $b_1+b_2 = \alpha(H)$.
\end{thm}

Theorems~\ref{thm:MainLow} and~\ref{thm:MainHigh} imply that a wide variety of graphs are $\sigma$-stable, including complete split graphs (a {\it complete split graph} is a graph of the form $K_r\vee \overline K_{t}$), complete bipartite graphs, friendship graphs, and odd cycles.
We will say that $H$ is {\it Type 1} if $2i^*-\nabla_{i^*}(H) \leq 2 \alpha(H)$, and {\it Type 2} if $2i^*-\nabla_{i^*}(H)=2\alpha(H)+1$. 
Of those graphs just listed, the complete split graphs and complete bipartite graphs are Type 1, and the odd cycles and friendship graphs are Type 2. 

\section{Technical lemmas}\label{sec:tech}

To prove Theorems \ref{thm:MainLow} and \ref{thm:MainHigh}, we will need several results from the literature on graphic sequences and some additional technical lemmas.
The first two results, due to Erd\H os and Gallai and to Kleitman and Wang, are simple characterizations of graphic sequences. 

\begin{thm}[Erd\H{o}s and Gallai~\cite{ErdGal}] \label{thm:ErdGal}
A nonincreasing sequence $\pi = (d_1, \ldots, d_n)$ of nonnegative integers is graphic if and only if $\sum_{i=1}^n d_i$ is even and, for all $p\in\{1,\ldots,n\}$,
\begin{equation*}
   \sum_{i=1}^p d_i\leq p(p-1)+\sum_{i=p+1}^n\min\{d_i,p\}. 
\end{equation*}
\end{thm}

\begin{thm}[Kleitman and Wang \cite{KleitWang}]\label{thm:KW}
Let $\pi = (d_1, \ldots, d_n)$ be a nonincreasing sequence of nonnegative integers, and let $i \in [n]$. If $\pi_i$ is the sequence defined by
\[
\pi_i =
\begin{cases}
(d_1 - 1, \ldots, d_{d_i}-1, d_{d_i+1}, \ldots, d_{i-1}, d_{i+1}, \ldots, d_n) & \text{ if } d_i < i \\
(d_1-1, \ldots, d_{i-1}-1, d_{i+1}-1, \ldots, d_{d_i+1}-1, d_{d_i+2}, \ldots, d_n) & \text{ if } d_i \geq i,
\end{cases}
\]
then $\pi$ is graphic if and only if $\pi_i$ is graphic.
\end{thm}

The process of deleting the term $d_i$ from the graphic sequence in the Kleitman-Wang Theorem (Theorem \ref{thm:KW}) and reducing the remaining terms accordingly is called {\it laying off} $d_i$.

The Kleitman-Wang Theorem has many useful consequences, some of which we list in the next result. 
For a graph $G$, let $\mathcal{D}^{(t)}(G)$ denote the family of subgraphs of $G$ that can be obtained by deleting exactly $t$ vertices from $G$. Alternatively, this is the family of induced subgraphs of $G$ with order $|V(G)|-t$. 
We say that a graphic sequence $\pi$ is potentially $\mathcal{D}\s{t}(G)$-graphic if there is a realization of $\pi$ that contains some graph in $\mathcal{D}\s{t}(G)$. 

\begin{cor}\label{KWCor}
Let $\pi_j$ be the sequence obtained from $\pi=(d_1, \ldots, d_n)$ by laying off the term $d_j$.
The following statements are true:
\begin{enumerate}
\item[(i)] There is a realization of $\pi$ in which a vertex of degree $d_j$ is adjacent to the $d_j$ vertices of highest degree, other than itself. 
\item[(ii)] If $\pi_j$ is potentially $H$-graphic, then $\pi$ is potentially $H$-graphic. 
\item[(iii)] If $\pi=(n-1, d_2, \ldots, d_n)$, then $\pi_1=(d_2-1, \ldots, d_n-1)$ is potentially $G$-graphic if and only if $\pi$ is potentially $(K_1\vee G)$ graphic.
Therefore $\pi$ is potentially $H$-graphic if and only if $\pi_1$ is potentially $\mathcal{D}^{(1)}(H)$-graphic.  
\end{enumerate}
\end{cor}

Part (i) of Corollary \ref{KWCor} guarantees the existence of a realization of $\pi$ in which the vertex of maximum degree, $d_1$, is adjacent to the next $d_1$ vertices of highest degree. 
Following terminology from~\cite{FHM}, we call such a realization a {\it canonical realization} of $\pi$.

A central part of the proofs of Theorems \ref{thm:MainLow} and \ref{thm:MainHigh} is the repeated laying off of terms from a graphic sequence.
The following lemma captures a few important facts about this process.  

\begin{lem}\label{lemma:iterated_KW}
Let $\pi$ be a graphic sequence of length $n$ with $\sigma(\pi)\ge mn$ and let $\pi_j$ be the graphic sequence of length $n-j$ obtained by iteratively laying off $j$ terms that are less than $\frac{m}{2}$, if this is possible.  The following hold:

\begin{enumerate}
\item[(i)] $\sigma(\pi_j)\ge m(n-j)+(m-2\CL{\frac m2}+2)j$,
\item[(ii)] the sum of the laid-off terms is at most $j(\left\lceil\frac{m}{2}\right\rceil-1)$.  
\end{enumerate}
\end{lem}

\begin{proof}
First observe that the greatest integer less than $\frac m2$ is $\CL{\frac m2}-1$.
When the term $d_i$ is laid off from $\pi$, the sum of the resulting sequence is $\sigma(\pi)-2d_i$.
Therefore, when we iteratively lay off $j$ terms, each of which is at most $\CL{\frac m2}-1$, then the sum of the terms that were laid off is at most $j(\CL{\frac m2}-1)$, and we obtain a sequence $\pi_j$ satisfying
\begin{equation*}
\sigma(\pi_j)\ge mn-2j\left(\left\lceil\frac{m}{2}\right\rceil-1\right)=m(n-j)+\left(m- 2\CL{\frac m2}+2\right)j.\qedhere
\end{equation*}
\end{proof}

\begin{lem}\label{PotDel2}
If $H$ is a graph of order $k$ and $t < k-\alpha(H)$, then there is an $F \in \mathcal{D}\s{t}(H)$ such that $\widetilde\sigma(F)\le\widetilde\sigma(H)-2t$.
Therefore, for such a graph $F$ and any $\delta> 0$, if $n$ is sufficiently large, then
$$\sigma(F,n)\le \sigma(H,n)-2tn+\delta n.$$
\end{lem}

\begin{proof}
Since $t < k-\alpha(H)$, we know $\alpha(H) < k-t$.
Thus, there is a graph $F \in \mathcal{D}\s{t}(H)$ with $\alpha(F)=\alpha(H)$, for example a subgraph of $H$ of order $k-t$ that contains a maximum independent set of $H$. For each $j \in \{\alpha(F)+1, \ldots, k-t\}$, we have $\nabla_j(F) \leq \nabla_j(H)$, because every $j$-vertex subgraph of $F$ is a $j$-vertex subgraph of $H$. This means that for each such $j$,
$$2j-\nabla_j(F) \geq 2j-\nabla_j(H) \ge 2i^*(H)-\nabla_{i^*(H)}(H).$$
In particular, we see that $2i^*(F)-\nabla_{i^*(F)}(F) \geq 2i^*(H)-\nabla_{i^*(H)}(H)$. 
By the definition of $\widetilde\sigma$, we have 
\begin{align*}
\widetilde\sigma(F)& = 2(k-t)-(2i^*(F)-\nabla_{i^*(F)}(F))-1 \\
 & \leq  2k-(2i^*(H)-\nabla_{i^*(H)}(H))-1 - 2t\\
 & = \widetilde\sigma(H,n)-2t.
\end{align*}

Let $\delta>0$.
By Theorem~\ref{FLMW}, we know that $\sigma(H,n)=\widetilde \sigma(H)n+o(n)$ and $\sigma(F,n)=\widetilde \sigma(F)n+o(n)$.
Thus there exists $n_{\delta}$ so that if $n>n_{\delta}$, then $\sigma(F,n)<\widetilde\sigma(F)n+\delta n$.
Therefore, for $n>n_\delta$,
\begin{align*}
\sigma(F,n)&<\widetilde \sigma(F)n+\delta n\\
&\le(\widetilde\sigma(H)-2t)n+\delta n\\
&\le \sigma(H,n)-2tn+\delta n.\qedhere
\end{align*}

\end{proof}

Given a graph $H$ with degree sequence $\pi(H)=(h_1, \ldots, h_k)$ and a graphic sequence $\pi=(d_1, \ldots, d_n)$, we say that $\pi$ is {\it degree-sufficient} for $H$ if $d_i \geq h_i$ for each $i$ with $1 \le i \le k$.  We will also require the following result from \cite{FLMW}.

\begin{thm}[The Bounded Maximum Degree Theorem (BMDT)]\label{BMDT}
Let $H$ be a graph of order $k$.
Let $n$ be sufficiently large and let $\pi = (d_1, \ldots, d_n)$ be a nonincreasing graphic sequence that satisfies the following:
\begin{enumerate}
\item[(i)] $\pi$ is degree sufficient for $H$, and 
\item[(ii)] $d_n \geq k-\alpha(H)$.
\end{enumerate}
There exists a function $f = f(\alpha(H),k)$ such that if $d_1 < n-f(\alpha(H),k)$, then $\pi$ is potentially $H$-graphic. 

\end{thm}

For our purposes, it is useful to know that the function $f$ in Theorem~\ref{BMDT} can be chosen as 
\begin{align*}
f(k) & =  \binom{k}{\FL{k/2}}(8k^2).
\end{align*}
While it is possible to further optimize the function by incorporating the mentioned dependence on $\alpha(H)$, that is not necessary here.  

Finally, we will need the following results that allow us to identify potentially $K_k$-graphic sequences and potentially $(K_r \vee \overline{K}_t)$-graphic sequences. 

\begin{thm}[Yin and Li \cite{YinLi}]\label{thm:YinLi}
Let $\pi = (d_1, \ldots, d_n)$ be a nonincreasing graphic sequence and let $k$ be a positive integer.
\begin{itemize}
\item[(i)] If $d_k \geq k-1$ and $d_i \geq 2(k-1) -i$ for $1 \leq i \leq k-2$, then $\pi$ is potentially $K_k$-graphic.
\item[(ii)] If $d_k \geq k-1$ and $d_{2k} \geq k-2$, then $\pi$ is potentially $K_k$-graphic.
\end{itemize}
\end{thm}

\begin{thm}[Yin \cite{Yin}]\label{yin_complete_split}
A nonincreasing sequence $\pi=(d_1, \ldots, d_n)$ is potentially $(K_r \vee \overline{K}_t)$-graphic if and only if there is a realization of $\pi$ containing a copy of $K_r \vee \overline{K}_t$ such that the vertices of the clique of order $r$ have degrees $d_1, \ldots, d_r$ and the vertices of the independent set of order $t$ have degrees $d_{r+1}, \ldots, d_{r+t}$. 
\end{thm}

\section{Proofs of Theorems \ref{thm:MainLow} and \ref{thm:MainHigh}}\label{sec:proofs}

Theorem \ref{thm:MainLow} is a relatively immediate corollary of Lemma~\ref{Const} below.   The proof of this lemma builds upon the approach used in \cite{FLMW}, although a different and much more careful analysis of the iterative process defined there is warranted to obtain our results.  Theorem \ref{thm:MainHigh} follows with a little more work. In particular, if $2i^*-\nabla_{i^*}(H)=2\alpha(H)+1$, then Lemma \ref{Const} may result in a realization of $\pi$ containing $K_{k-\alpha(H)-1} \vee \overline{K}_{\alpha(H)+1}$.
In this case, further analysis is necessary to show that $H$ is $\sigma$-stable. 

\begin{lem}\label{Const}
Let $H$ be a graph of order $k$ and independence number $\alpha$. Let $\epsilon$ be given, where $0<\epsilon<\frac 12$.
There exists $n_0= n(\epsilon, H)$ and $\delta >0$ such that for any graphic sequence $\pi$ of length $n \geq n_0$ satisfying $\sigma(H,n)\ge \sigma(\pi) \geq \sigma(H,n) -\delta n$, either
\begin{enumerate}
\item[(i)] $\pi$ is potentially $H$-graphic, 
\item[(ii)] $\| \pi-\pi' \| < \epsilon n$ for some $\pi' \in \mathcal{P}(H,n)$, or 
\item[(iii)] $\pi$ is potentially $(K_{k-\alpha-1} \vee \overline{K}_{\alpha+1})$-graphic.  
\end{enumerate}
\end{lem}
\medskip
{\bf Proof Sketch:} Before proving Lemma~\ref{Const}, we provide an outline of the argument.
The proof proceeds as follows. We first initialize and then state an algorithm that iterates a process of (a) deleting specific vertices from a realization of the current sequence and taking the degree sequence of the resulting graph, and (b) laying off small terms from the resulting degree sequence.  

The algorithm performs an iteration only if the leading term of the current sequence is too large to satisfy the condition on $d_1$ given in Theorem \ref{BMDT}, the BMDT.  In this case, we begin by deleting the nonneighbors of the vertex of maximum degree in a canonical realization of the sequence. This yields a graph with a dominating vertex, and that dominating vertex is then deleted.

The algorithm then takes the degree sequence of the resulting graph and repeatedly lays off small terms so as to raise the minimum term to meet condition (ii) of the BMDT.
If too many small terms are laid off, then we can show that $\pi$ is potentially $H$-graphic, yielding conclusion~(i) of the lemma.  If relatively few terms are laid off and the maximum term is still very large, then the process iterates again.  The algorithm halts when it generates a sequence with a sufficiently small maximum term or when it has iterated a given number of times.

If the algorithm hits its halting condition without first yielding conclusion~(i), we analyze the possible realizations of the resulting sequence.  A consequence of the algorithm is that we have a realization of $\pi$ containing a large complete split graph $S$ of order $(1-o(1))n$, where the vertices of the clique in $S$ are the dominating vertices that are deleted through each iteration of the algorithm.  

 Suppose that there are $\ell$ vertices in the dominating clique of $S$, or equivalently that the algorithm was iterated $\ell$ times.  It therefore suffices to modify the realization of $\pi$ to construct any $(k-\ell)$-vertex induced subgraph of $H$ within the independent set of $S$.  If it is not possible to construct such a graph, then we show that conclusion~(ii) or (iii) follows.\\

{\bf Proof of Lemma \ref{Const}.}  
Let $\delta<\frac{\epsilon}{16k^3+48k^2+(32+\epsilon)k}$.
For sufficiently large $n$, let $\pi=(d_1,\ldots,d_n)$ be a graphic sequence of length $n$ such that $\sigma(\pi) \geq (2(k-i^*)+\nabla_{i^*}-1-\delta)n$.
First note that if $d_{2k}\ge k-1$, then by statement (ii) of Theorem~\ref{thm:YinLi} it follows that $\pi$ is potentially $K_k$-graphic, and hence potentially $H$-graphic.
Therefore for the remainder of the proof we assume that $d_{2k}\le k-2$.
It follows from this assumption that if $\pi'=(d_1',\ldots,d_{n'}')$ is a sequence obtained by iteratively laying terms off of $\pi$, or is the degree sequence of a subgraph of a realization of $\pi$, then $d_{2k}'\le k-2$.

Define
\[
b_H=
\begin{cases}
0 &\text{ if }H\text{ is Type 1}\\
1 &\text{ if }H\text{ is Type 2}.
\end{cases}
\]
Note that, for any choice of $H$, 
$$\CL{\frac{\tilde{\sigma}(H)}{2}}=\CL{\frac{(2(k-i^*)+\nabla_{i^*}-1)}{2}}=\CL{\frac{2k-(2i^*-\nabla_{i^*})-1}{2}}\ge k- \alpha-b_H.$$

Prior to beginning the algorithm described above, we perform an initialization step, in which we apply the Kleitman-Wang Theorem (Theorem~\ref{thm:KW}) to $\pi$ to iteratively lay off terms that are at most $\CL{\frac{\sigma(\pi)}{2n}}-1$, as long as such terms exist.
Let $\pi^{(0)}$ be the resulting sequence, and note that $\pi^{(0)}$ has minimum term at least $\CL{\frac{\sigma(\pi)}{2n}}$.
Since $\delta<\frac{1}{2k}$ and all other terms are integers, 
\begin{align*}
\CL{\frac{\sigma(\pi)}{2n}}&\ge \CL{\frac{(2(k-i^*)+\nabla_{i^*}-1-\delta)}{2}}\\
&= \CL{\frac{(2(k-i^*)+\nabla_{i^*}-1)}{2}}\\
&= \CL{\frac{\tilde \sigma (H)}{2}}\\
&\ge k-\alpha-b_H.
\end{align*}
Thus $\pi^{(0)}$ has minimum term at least $k-\alpha-b_H$, and each term that is laid off in the initialization step has value at most $\CL{\frac{\sigma(\pi)}{2n}}-1=\CL{\frac{\tilde\sigma(H)}{2}}-1$.
\medskip



Let $j$ be the number of terms that are laid off to obtain $\pi^{(0)}$.
Since each of these terms is at most $\CL{\frac{\tilde\sigma(H)}{2}}-1$, Lemma~\ref{lemma:iterated_KW} implies that 
\begin{equation}\label{eq:pi0}
\sigma(\pi^{(0)}) \geq (2(k-i^*)+\nabla_{i^*}-1-\delta)(n-j) + (1-\delta)j.
\end{equation} 
If $(1-\delta)j > 2\delta(n-j)$, or equivalently $j>\frac{2\delta}{1+\delta}n$, then
$$\sigma(\pi^{(0)}) \geq (2(k-i^*)+\nabla_{i^*}-1)(n-j) + \delta(n-j) .$$
If we fix $j$ to be the smallest integer that is greater than $\frac{2\delta}{1+\delta}n$, then if $n$ is sufficiently large, so too is $n-j$.  In this case, Theorem \ref{FLMW} implies that the sequence obtained after laying off $j$ terms is potentially $H$-graphic.  By Corollary \ref{KWCor} $\pi$ is also potentially $H$-graphic, yielding conclusion~(i).
Thus, we assume for the rest of the proof that $\pi^{(0)}$ is obtained after laying off at most $\frac{2\delta}{1+\delta}n$ terms.

After the initialization, we perform the iteration described below to create a new sequence, $\pi\s{1}$. Successive iterations create a family of sequences $\pi\s{1}, \ldots, \pi\s{\ell}$ where $\ell$ is in $\{1,\ldots,k-\alpha-b_H-1\}$.
For each $t \ge 0$, let $\pi\s{t}=(d_1^{(t)}, \ldots, d_{n_t}^{(t)})$ be the sequence  that results from the $t^{\rm th}$ iteration of the algorithm,
and let $R_t$ be a canonical realization of $\pi\s{t}$ on the vertex set $\{v_1^{(t)}, \ldots, v_{n_t}^{(t)}\}$ where $d(v_j^{(t)}) = d_j^{(t)}$, as guaranteed by Corollary \ref{KWCor}.
\medskip

\noindent{\bf Algorithm Iteration:} Starting with $t=0$, proceed as follows:
\begin{enumerate}[{Step}~1:]
\item\label{Alg:halt} If $d_1^{(t)}<n_{t}-\binom{k}{\fl{k/2}}(8k^2)$ or $t=k-\alpha-b_H$, then set $\ell=t$ and halt.
Otherwise, proceed to Step~\ref{Alg:deletesmall}.
\item\label{Alg:deletesmall} Remove the non-neighbors of the vertex $v_1^{(t)}$ from $R_t$ to obtain a graph $\widehat R_t$.
Let $\widehat{n_t}$ denote the number of vertices in $\widehat R_t$ and let $\widehat\pi\s{t}= \pi(\widehat R_t)$. 
\item\label{Alg:deletedom} Lay off the largest term of $\widehat\pi\s{t}$, which is necessarily $\widehat{n_t}-1$, and call the resulting sequence $\widecheck\pi\s{t}$.
\item\label{Alg:itKW} As in the initialization step, repeatedly apply Theorem \ref{thm:KW} to $\widecheck \pi\s{t}$, laying off minimum terms until we obtain a sequence in which each term is at least $$k-i^*+\left\lceil\frac{\nabla_{i^*}-1-(t+1)\delta}{2}\right\rceil-(t+1),$$  which, as $\delta$ is sufficiently small, is at least $$k-(t+1)-\alpha-b_H.$$
If at least $\frac{(t+3)\delta}{1-k\delta}n_t$ terms are laid off from $\widecheck \pi\s t$ in this step, then halt.  In Claim~\ref{claim:step3} we will show that in this case, $\pi$ is potentially $H$-graphic.  Otherwise, let $\pi\s{t+1}$ be the sequence that results from this step, set $t=t+1$, and return to Step~\ref{Alg:halt}.
\end{enumerate}



Before beginning a more in-depth analysis of the sequences generated by the algorithm, we give the following observations. 

\begin{itemize}
\item[(a)]\label{s1obs:removed} At most $\binom{k}{\fl{k/2}}(8k^2)$ vertices are removed in each iteration of Step~\ref{Alg:deletesmall}.
This follows because the algorithm halts in Step~\ref{Alg:halt} if $d_1^{(t)}<n_{t}-\binom{k}{\fl{k/2}}(8k^2)$.
\item[(b)]\label{s1obs:bound} Every vertex deleted in Step~\ref{Alg:deletesmall} has degree at most $k-2$.
This follows from Theorem~\ref{thm:YinLi}; the nonneighbors of $v_1^{(t)}$ have degrees $d_{d_1\s t+2}, \ldots,d_{n_t}$, and 
for sufficiently large $n_t$, $d_1\s t+2\ge n_t-\binom{k}{\FL{k/2}}(8k^2)+2\ge 2k$. Thus, the deleted vertices must have degree at most $k-2$.
\item[(c)]\label{s1obs:conical} After Step~\ref{Alg:deletesmall}, $v_1\s t$ is a dominating vertex in $\widehat R_t$.
Laying off the largest term of $\widehat \pi\s{t}$ in Step~\ref{Alg:deletedom} is equivalent to removing a dominating vertex from $\widehat R_t$, so $\widecheck\pi\s{t}$ is the degree sequence of the graph $\widehat R_t-v_1\s t$.
\end{itemize}

Next, we provide some properties of the sequences $\pi\s{t}$ generated by the algorithm that will be useful if we are to show $\pi$ is potentially $H$-graphic.  The first also appears in \cite{FLMW}, so we omit the proof here in the interest of concision.

\begin{claim}\label{Cl:graphbuild} 
For each $t\in\{0,\ldots,\ell\}$, if $G\s{t}$ is a realization of $\pi\s{t}$, then $\pi$ is potentially $K_t \vee G\s t$-graphic.  Consequently, if $\pi\s{t}$ is potentially $\mathcal{D}\s{t}(H)$-graphic, then $\pi$ is potentially $H$-graphic. 
\end{claim}

%
%

The next claim gives a reasonably large lower bound on $\sigma(\pi\s t)$.

\begin{claim}\label{Cl:SumIter} 
$\sigma(\pi\s{t}) \geq (2(k-i^*)+\nabla_{i^*}(H) -1- (t+1)\delta -2t)n_t$.  
\end{claim}

\begin{proof}[Proof of Claim \ref{Cl:SumIter}]
We proceed by induction on $t$. 
For $t=0$, the claim follows from Lemma~\ref{lemma:iterated_KW}, as observed in Equation~(\ref{eq:pi0}).

Let $t \ge 1$ and assume that $\sigma(\pi\s{t-1}) \geq (2(k-i^*)+\nabla_{i^*}(H) -1- ((t-1)+1)\delta -2(t-1))n_{t-1}$. 
Observations (a) and (b) above show that Step~\ref{Alg:deletesmall} deletes at most $\binom{k}{\FL{k/2}}(8k^2)$ vertices each of degree at most $k-2$.
Therefore $\sigma(\widehat \pi \s {t-1})\ge \sigma(\pi\s{t-1})-\binom{k}{\FL{k/2}}(16k^3)$.
To obtain $\widecheck \pi\s {t-1}$ from $\widehat \pi \s {t-1}$, we lay off the maximum term, which is at most $n_{t-1}-1$.
Therefore $\sigma(\widecheck\pi\s {t-1})\ge \sigma(\pi\s{t-1})-\binom{k}{\FL{k/2}}(16k^3)-2(n_{t-1}-1)$.
Let $\widecheck n_{t-1}$ denote the number of terms in $\widecheck \pi\s {t-1}$.
Observe that $n_{t-1}-\binom{k}{\FL{k/2}}(8k^2)-1\le\widecheck n_{t-1}< n_{t-1}$.
Since $\widecheck n_{t-1}$ grows with $n_{t-1}$ while $k$ and $\delta$ are fixed, if $\widecheck n_{t-1}$ is sufficiently large, then $\delta\widecheck n_{t-1}>\binom{k}{\FL{k/2}}(16k^3)$.
Therefore
\begin{align*}
\sigma(\widecheck\pi\s {t-1})&\ge \sigma(\pi^{(t-1)})-2(n_{t-1}-1)-\binom{k}{\FL{k/2}}(16k^3)\\
&\ge (2(k-i^*)+\nabla_{i^*}(H) -1- t\delta -2(t-1))n_{t-1}-2(n_{t-1}-1) -\delta\widecheck n_{t-1}\\
&\ge (2(k-i^*)+\nabla_{i^*}(H) -1- t\delta -2t)n_{t-1} -\delta\widecheck n_{t-1}\\
&\ge (2(k-i^*)+\nabla_{i^*}(H) -1- t\delta -2t)\widecheck n_{t-1} -\delta\widecheck n_{t-1}\\
&\ge (2(k-i^*)+\nabla_{i^*}(H) -1- (t+1)\delta -2t)\widecheck n_{t-1}.
\end{align*}

Finally, to obtain $\pi\s{t}$ from $\widecheck \pi\s {t-1}$, we iteratively lay off terms, each of which is at most $k-i^*+\frac{\nabla_{i^*}(H)-1-(t+1)\delta}{2}-t$.
Let $j$ denote the number of these terms that are laid off to obtain $\pi_t$; note that $n_t=\widecheck n_{t-1}-j$.
By Lemma~\ref{lemma:iterated_KW}, it follows that
\begin{equation*}\sigma(\pi_t) \geq (2(k-i^*)+\nabla_{i^*}(H)-1- (t+1)\delta -2t)n_{t}.\qedhere
\end{equation*}
\end{proof}
 
Next, we show that it is valid to declare that $\pi$ is potentially $H$-graphic if too many terms are laid off in Step~\ref{Alg:itKW}.

\begin{claim}\label{claim:step3}
If at least $\frac{(t+3)\delta}{1-k\delta}n_t$ terms are laid off from $\widecheck \pi\s t$ to obtain $\pi\s {t+1}$ in Step~\ref{Alg:itKW} of the algorithm, then $\pi$ is potentially $H$-graphic.
\end{claim}

\noindent\textit{Proof of Claim \ref{claim:step3}.}  As shown in the proof of Claim~\ref{Cl:SumIter}, 
$$\sigma(\widecheck\pi\s{t}) \geq (2(k-i^*)+\nabla_{i^*}(H)-1-(t+2)\delta -2(t+1))\widecheck n_t.$$
Let $j$ denote the number of terms that are laid off from $\widecheck \pi_t$ to obtain $\pi\s {t+1}$, and assume that $j\ge \frac{(t+3)\delta}{1-k\delta}n_t$.
Recall that each of those terms is at most $k-i^*+\frac{\nabla_{i^*}(H)-1-(t+2)\delta}{2}-(t+1)$.
Since $\nabla_{i^*}$ is an integer and $\delta<\frac{1}{t+2}$, it follows from Lemma~\ref{lemma:iterated_KW} that
\[
\sigma(\pi\s{t+1})\ge [2(k-i^*)+\nabla_{i^*}(H)-1- (t+2)\delta -2(t+1)](\widecheck n_t-j) + (1-(t+2)\delta)j.
\]
Recall that $t< k-\alpha-b_H$, or else the algorithm would have terminated, and therefore $k\ge t+2$.
Thus 
\begin{align*}
j&\ge \frac{(t+3)\delta}{1-k\delta}n_t\\
&\ge \frac{(t+3)\delta}{1-k\delta}(\widecheck n_t-j)\\
&\ge \frac{(t+3)\delta}{1-(t+2)\delta}(\widecheck n_t-j)\\
\end{align*}

Consequently, if $\widecheck n_t$ is sufficiently large, we have that,
\begin{align*}
\sigma(\pi\s{t+1})&\ge [2(k-i^*)+\nabla_{i^*}(H)-1- (t+2)\delta -2(t+1)](\widecheck n_t-j)+(t+3)\delta(\widecheck n_t-j)\\
&\ge[2(k-i^*)+\nabla_{i^*}(H)-1](\widecheck n_t-j)-2(t+1)(\widecheck n_t-j)+\delta(\widecheck n_t-j)\\
&=\sigma(H,\widecheck n_t-j)-2(t+1)(\widecheck n_t-j)+\delta(\widecheck n_t-j).
\end{align*}
By Lemma~\ref{PotDel2}, there exists some $F \in \mathcal{D}\s{t+1}(H)$ such that $\sigma(\pi\s{t+1}) \geq \sigma(F, \widecheck n_t -j)$.
Therefore $\pi\s{t+1}$ is potentially $\mathcal{D}\s{t+1}(H)$-graphic. 
By Claim \ref{Cl:graphbuild}, $\pi$ is potentially $H$-graphic. \hfill$\Box$\\

We conclude our discussion of the algorithm and the intermediate sequences generated by it by showing that it will not iterate too many times, relative to $n$.  

\begin{claim}\label{Cl:n-nl}
If the algorithm terminates in Step~\ref{Alg:halt} by generating $\pi\s\ell$, then $n-n_\ell<\frac\epsilon{8k} n$.
\end{claim}

\begin{proof}[Proof of Claim~\ref{Cl:n-nl}]
In generating $\pi\s\ell$, Steps~\ref{Alg:deletesmall}, \ref{Alg:deletedom}, and \ref{Alg:itKW} are each iterated at most $k-\alpha-1$ times.
In each iteration of Step~\ref{Alg:deletesmall}, at most $\binom{k}{\fl{k/2}}(8k^2)$ vertices are deleted.
In each iteration of Step~\ref{Alg:deletedom}, one term is laid off. 
Finally, Claim~\ref{claim:step3} shows that in each iteration of Step~\ref{Alg:itKW}, at most $\frac{(3+k)\delta}{1-k\delta}n$ terms are laid off.
Since the initialization step lays off at most $\frac{2\delta}{1+\delta}n$ terms, we conclude that 
\begin{align*}
n-n_\ell & \leq (k-\alpha-1)\left(\binom{k}{\fl{k/2}}(8k^2)+1+\frac{(3+k)\delta}{1-k\delta}n\right)+\frac{2\delta}{1+\delta}n \\
& \leq  \frac{\delta(k^2+3k+2)}{1-k\delta}n+k\binom{k}{\FL{k/2}}(8k^2)+k.
\end{align*}
Because $\delta<\frac{\epsilon}{16k^3+48k^2+(32+\epsilon)k}$, it follows that $\frac{\delta(k^2+3k+2)}{1-k\delta}< \frac{\epsilon}{16k}$.
Furthermore, if $n$ is sufficiently large, then $k\binom{k}{\FL{k/2}}(8k^2)+k<\frac{\epsilon}{16k}n$, so that
\begin{equation*}
n-n_\ell\le \frac{\epsilon}{8k}n.\qedhere
\end{equation*}
\end{proof}
%
%

Next, we analyze the realizations of $\pi\s\ell$ with the goal of showing that some realization of $\pi$ contains a supergraph of $H$, and therefore that $\pi$ is potentially $H$-graphic.   Recall that the algorithm halts either when it has iterated $k-\alpha-b_H$ times, or when the sequence $\pi\s{t}$ satisfies the maximum degree condition of Theorem \ref{BMDT}.

Suppose that the algorithm halts because it has iterated $k-\alpha-b_H$ times, so that $\ell=k-\alpha-b_H$.
Note that $\pi\s\ell$ is trivially $(\overline{K}_{\alpha+b_H})$-graphic, so Claim~\ref{Cl:graphbuild} implies that $\pi$ is potentially $(K_{k-\alpha-b_H} \vee \overline{K}_{\alpha+b_H})$-graphic.
If $H$ is Type 1, then $b_H=0$ and since $K_{k-\alpha} \vee \overline{K}_{\alpha}$ is a supergraph of $H$, it follows that $\pi$ is potentially $H$-graphic.
If $H$ is Type 2, then $b_H=1$ and $\pi$ is potentially $(K_{k-\alpha-1} \vee \overline{K}_{\alpha+1})$-graphic. 
Therefore $\pi$ satisfies conclusion (i) or (iii) of Lemma~\ref{Const}. 

Assume then that the algorithm stops when $t= \ell  < k-\alpha-b_H$ and $d_1^{(\ell)} < n_\ell-\binom{k}{\fl{k/2}}(8k^2)$.
Observe that $d_{n_\ell}^{(\ell)}$, the minimum term of $\pi\s{\ell}$, satisfies $d_{n_\ell}^{(\ell)}\ge k-\ell-\alpha(H)-b_H$ as the algorithm increments $t$ one more time before declaring the value of $\ell$. 
If $\pi\s\ell$ is degree-sufficient for $S_\ell = K_{k-\ell-\alpha-b_H} \vee \overline{K}_{\alpha+b_H}$, then the BMDT implies that $\pi\s\ell$ is potentially $S_\ell$-graphic. 
Since $K_\ell \vee S_\ell = K_{k-\alpha-b_H} \vee \overline{K}_{\alpha +b_H}$, it follows that $\pi$ is potentially $H$-graphic if $H$ is Type 1, and $\pi$ is potentially $(K_{k-\alpha-1} \vee \overline{K}_{\alpha+1})$-graphic if $H$ is Type 2.
Therefore, again $\pi$ satisfies conclusion (i) or (iii) of Lemma~\ref{Const}.

We therefore assume for the remainder of the proof that $\pi\s{\ell}$ is not degree-sufficient for $S_\ell$. 
Let $p = \max \{j: d_j^{(\ell)} \geq k-\ell-1\}$. 
Since $\pi\s{\ell}$ is not degree-sufficient for $S_\ell$ but the minimum term of $\pi\s{\ell}$ is at least the minimum degree of $S_\ell$, namely $k-\ell-\alpha-b_H$, we know that $p < k-\ell-\alpha-b_H$. 
Thus, $\pi\s{\ell}$ is degree-sufficient for $K_p \vee \overline{K}_{k-\ell-p}$. 

For $j \in \{\alpha+1, \ldots, k\}$, let $F_j$ be a $j$-vertex induced subgraph of $H$ such that $\Delta(F_j) = \nabla_j(H)$.
We conclude the proof of Lemma~\ref{Const} with Claims \ref{claim:degree_sufficient} and \ref{claim:superfunediting} and a final observation.  Our first claim is embedded within the proof of the Theorem 3 in \cite{FLMW}, so we omit the proof here.  

\begin{claim}\label{claim:degree_sufficient}
If $\pi\s{\ell}$ is degree-sufficient for $K_p \vee F_{k-\ell-p}$, then $\pi$ is potentially $H$-graphic.
\end{claim}

Next we show that if the hypothesis of Claim \ref{claim:degree_sufficient} does not hold, then $\|\pi-\wtilpi_{k-\ell-p}(H,n)\|$ is small.  

\begin{claim}\label{claim:superfunediting}
If $\pi\s\ell$ is not degree-sufficient for $K_p \vee F_{k-\ell-p}$, then $\|\pi-\wtilpi_{k-\ell-p}(H,n)\| < \epsilon n$.
\end{claim}

\begin{proof}[Proof of Claim \ref{claim:superfunediting}]
Let $\eta=(a_1,\ldots,a_{n_\ell+\ell})$ be the degree sequence of $K_\ell \vee G\s\ell$, where $G\s\ell$ is a realization of $\pi\s{\ell}$.  We compute upper bounds for 
\begin{itemize}
\item $\| \pi - \eta\|$, 
\item $\|\eta - \wtilpi_{k-\ell-p}(H, n_\ell + \ell)\|$, and 
\item $\| \wtilpi_{k-\ell-p}(H, n_\ell+\ell) - \wtilpi_{k-\ell-p}(H,n)\|$, 
\end{itemize}
and apply the triangle inequality to show that $\|\pi-\wtilpi_{k-\ell-p}(H,n)\|< \epsilon n$.

First we bound $\| \pi - \eta \|$.
By Claim~\ref{Cl:graphbuild}, $\pi$ is potentially $(K_\ell \vee G\s\ell)$-graphic, and hence $d_j\ge a_j$ for all $j\in\{1,\ldots,n\}$.
Therefore $\| \pi - \eta \|=\sum_{j=1}^nd_j-a_j=\sigma(\pi)-\sigma(\eta)$.
From the assumptions of Lemma~\ref{Const}, we know that $\sigma(\pi)\le \sigma(H,n)$.
Hence by Theorem~\ref{FLMW}, for $n$ sufficiently large,
$$\sigma(\pi)< (2(k-i^*)+\nabla_{i^*}-1)n+\delta n.$$

Now we seek a lower bound on $\sigma(\eta)$.
By Claim~\ref{Cl:SumIter}, we know that $\sigma(\pi\s{\ell}) \ge (2(k-i^*)+\nabla_{i^*}-1-(\ell+1)\delta -2\ell)n_\ell$.
By the definition of $\eta$, 
\begin{align}
\sigma(\eta) & = \sigma(\pi\s{\ell}) + (n_\ell+\ell-1)\ell+\ell n_\ell\nonumber\\
& \geq (2(k-i^*)+\nabla_{i^*}-1-(\ell+1)\delta -2\ell)n_{\ell} + 2\ell n_{\ell} + \ell^2 - \ell\nonumber\\
& \geq (2(k-i^*)+\nabla_{i^*}-1-(\ell+1)\delta)n_{\ell}.\label{etalb}
\end{align}
Therefore
\begin{align*}
\| \pi - \eta \|&=\sigma(\pi)-\sigma(\eta)\\
&< (2(k-i^*)+\nabla_{i^*}-1)n+\delta n- (2(k-i^*)+\nabla_{i^*}-1-(\ell+1)\delta)n_{\ell}\\
&= (2(k-i^*)+\nabla_{i^*}-1)(n-n_\ell)+\delta n+(\ell+1)\delta n_\ell\\
&\le 3k(n-n_{\ell})+(\ell+2)\delta n.
\end{align*}
Recall from Claim~\ref{Cl:n-nl} that
$$n-n_\ell<\frac{\epsilon}{8k}n.$$
Since $\delta<\frac{\epsilon}{8(\ell+2)}$, we conclude that
\begin{equation}\label{eq:pi-eta}
\| \pi - \eta \|< \frac{4\epsilon}{8}n = \frac{\epsilon}{2}n.
\end{equation}

Next we bound $\|\eta - \wtilpi_{k-\ell-p}(H,n_\ell + \ell)\|$.  We first observe that
\begin{align*}
\sigma(\wtilpi_{k-\ell-p}(H,n_\ell + \ell))&=(2(\ell+p)+\nabla_{k-\ell-p}-1)(n_\ell+\ell)-(\ell+p)(\ell+p+\nabla_{k-\ell-p}).
\end{align*}
Now we establish a lower bound on $\sigma(\eta)$.
Since $p<k-\ell-\alpha-b_H$ and $b_H$ is either 0 or 1, it follows that $k-\ell-p\ge \alpha+1$, so that $2(k-i^*)+\nabla_{i^*}\ge 2(\ell+p)+\nabla_{k-\ell-p}$.

Continuing from inequality~(\ref{etalb}) and using the fact that $\ell+p\le k$, we have
\begin{align}
\sigma(\eta) & \geq (2(k-i^*)+\nabla_{i^*}-1-(\ell+1)\delta)n_{\ell}\nonumber\\
&\ge (2(\ell+p)+\nabla_{k-\ell-p}-1)n_{\ell}-(\ell+1)\delta n_{\ell}\nonumber\\
&\ge \sigma(\wtilpi_{k-\ell-p}(H,n_\ell + \ell)) \nonumber\\ &\qquad\qquad- (2(\ell+p) + \nabla_{k-\ell-p} - 1)\ell+  (\ell+p)(\ell+p+\nabla_{k-\ell-p}) -(\ell+1)\delta n_{\ell}\nonumber   \\
&\ge \sigma(\wtilpi_{k-\ell-p}(H,n_\ell + \ell)) -3k^2-(\ell+1)\delta n_{\ell}.\label{etalower}
\end{align}

To establish an upper bound on $\sigma(\eta)$, we first observe that it follows from the definition of $p$ that $d_j^{(\ell)} \leq k-\ell-2$ for $j\in\{p+1,\ldots,k-\ell-1\}$.
Also, $d_j^{(\ell)} \leq \nabla_{k-\ell-p} +p-1$ for $j \geq k-\ell$ since $\pi\s\ell$ is not degree-sufficient for $K_p \vee F_{k-\ell-p}$. 
Therefore, the only terms of $\eta=(a_1, \ldots, a_{n_\ell+\ell})$ that could exceed the corresponding terms of $\wtilpi_{k-\ell-p}(H, n_{\ell}+\ell)$ are those $a_j$ for which $p+\ell+1 \leq j \leq k-1$. 
Since the largest terms of $\eta$ in that range are at most $k-2$, a term of $\eta$ can exceed its corresponding term in $\wtilpi_{k-\ell-p}(H,n_\ell+\ell)$ by at most $k-2$.
Therefore we have the following upper bound on $\sigma(\eta)$:
\begin{equation}\label{etaupper}
\sigma(\eta) < \sigma(\wtilpi_{k-\ell-p}(H,n_\ell+\ell))+k^2.
\end{equation}
From inequalities~\eqref{etalower} and~\eqref{etaupper}, it follows that
\begin{equation*}
|\sigma(\wtilpi_{k-\ell-p}(H,n_\ell+\ell))-\sigma(\eta)|\le (\ell+1)\delta n_{\ell}+3k^2.
\end{equation*}
Because $\delta<\frac{\epsilon}{32(\ell+1)}$, and $n_\ell$ is sufficiently large, we conclude that  
$$|\sigma(\wtilpi_{k-\ell-p}(H,n_\ell+\ell))-\sigma(\eta)|<\frac{\epsilon}{16}n_\ell\le\frac{\epsilon}{16}n.$$

Thus, $\|\eta - \wtilpi_{k-\ell-p}(H,n_\ell + \ell)\|$ is bounded from above by $|\sigma(\wtilpi_{k-\ell-p}(H,n_\ell+\ell))-\sigma(\eta)|$ plus twice the sum of the terms in $\eta$ that exceed their corresponding terms in $\wtilpi_{k-\ell-p}(H,n_\ell + \ell)$.
As observed in the discussion preceding inequality (\ref{etaupper}), the sum of such terms is less than $k^2$.
Since $2k^2<\frac{\epsilon}{16}n$ for sufficiently large $n$,  we conclude that
\begin{equation}\label{eq:eta-wtilpi}
\|\eta - \wtilpi_{k-\ell-p}(H,n_\ell + \ell)\| \leq |\sigma(\wtilpi_{k-\ell-p}(H, n_{\ell}+\ell))-\sigma(\eta)| + 2k^2 <\frac{\epsilon}{8}n.
\end{equation}

Finally, direct computation shows that 
\begin{align}
\| \wtilpi_{k-\ell-p}(H, n_{\ell}+\ell)-  \wtilpi_{k-\ell-p}(H, n)\| &= (n-(n_\ell +\ell))(2(\ell+p)+\nabla_{k-\ell-p}-1)\nonumber\\
&\le 3k(n-n_\ell)\nonumber\\
&< 3k\left(\frac{\epsilon}{8k}n\right) =\frac{3\epsilon}{8}n. \label{wtilpi-wtilpi}
\end{align}

The proof of Claim~\ref{claim:superfunediting} now follows from the triangle equality applied to  inequalities (\ref{eq:pi-eta}), (\ref{eq:eta-wtilpi}), and~(\ref{wtilpi-wtilpi}):
\begin{align*}
\|\pi-\wtilpi_{k-\ell-p}(H,n)\| & \leq  \|\pi-\eta\|+\|\eta - \wtilpi_{k-\ell-p}(H, n_{\ell}+\ell)\| \\
 & \qquad \qquad \qquad \qquad+ \| \wtilpi_{k-\ell-p}(H, n_{\ell}+\ell)-  \wtilpi_{k-\ell-p}(H, n)\|\\
&< \frac{\epsilon}{2}n+\frac{\epsilon}{8}n+\frac{3\epsilon}{8}n\\
&=\epsilon n.\qedhere
\end{align*}
\end{proof}

To finish the proof of Lemma~\ref{Const}, it remains to show that $\wtilpi_{k-\ell-p}(H,n)$ is in ${\mathcal P}(H,n)$ so that Claim~\ref{claim:superfunediting} implies conclusion (ii) of the lemma.
Since 
$|\sigma(\wtilpi_{k-\ell-p}(H,n))-\sigma(\pi)|\le \|\pi-\wtilpi_{k-\ell-p}(H,n)\|<\epsilon n,$
and $|\sigma(\pi)-\sigma(H,n)|<\delta n$, it follows that $|\sigma(\wtilpi_{k-\ell-p}(H,n))-\sigma(H,n)|<(\epsilon+\delta)n$.
Since $\delta<\epsilon<\frac12$ and the coefficient on the linear term of $\sigma(\wtilpi_i)$ is an integer for all $i$, this implies that that $\wtilpi_{k-\ell-p}\in{\mathcal P}(H,n).$  This completes the proof of Lemma \ref{Const}.\hfill $\Box$\\

When $H$ is Type 1, the proof of Lemma~\ref{Const} shows that we only reach conclusions~(i) or (ii)~of Lemma~\ref{Const}, implying that $H$ is $\sigma$-stable. Theorem~\ref{thm:MainLow} therefore follows. 
When $H$ is Type 2, however, Lemma~\ref{Const} may only guarantee a realization of $\pi$ containing $K_{k-\alpha(H)-1} \vee \overline{K}_{\alpha(H)+1}$.
To prove Theorem~\ref{thm:MainHigh}, we further analyze this case. 
We start with Theorem~\ref{thm:NotStable} below, which gives one condition under which a Type 2 graph is not $\sigma$-stable. 
Recall that $S_{b_1, b_2}$ is the double star with central vertices of degree $b_1+1$ and $b_2+1$.

\begin{thm}\label{thm:NotStable}
If $H$ is a graph of order $k$ and independence number $\alpha$ such that $2i^*-\nabla_{i^*}(H) = 2\alpha+1$, 
and $H \not\subseteq K_{k-\alpha-2} \vee S_{b_1, b_2}$ for any $b_1$ and $b_2$ with $b_1+b_2 = \alpha$, then $H$ is not $\sigma$-stable. 
\end{thm}

\begin{proof}
Consider the sequence 
\[\rho=\rho(H,n)=\left((n-1)^{k-\alpha-2}, \CL{\frac{n+k-\alpha-2}{2}},\FL{\frac{n+k-\alpha-2}{2}}, (k-\alpha-1)^{n-k+\alpha}\right),\]
where $\alpha=\alpha(H)$. 
This is the degree sequence of the graph $G=K_{k-\alpha-2} \vee S_{\CL{\frac{n-k+\alpha}{2}},\FL{\frac{n-k+\alpha}{2}}}$, and $G$ is the only realization of $\rho$.
Note that $\sigma(\rho)=2(k-\alpha-1)n-(k+\alpha)(k-\alpha-1)$. 

By assumption, $H$ is Type 2 so $\sigma(H,n)=2(k-\alpha-1)n + o(n)$.
Thus $\sigma(\rho) \geq \sigma(H,n)-\delta n$ for any $\delta$, provided that $n$ is large enough. 
Consider a set $X$ of $k$ vertices in $G$.
If $X$ contains the $k-\alpha$ vertices of highest degree, then $G[X]=K_{k-\alpha-2} \vee S_{b_1, b_2}$ for some $b_1+b_2=\alpha$.
Otherwise, $G[X]$ contains at least $\alpha+1$ of the vertices of degree $k-\alpha-1$ in $G$, which are pairwise nonadjacent, and $\alpha(G[X])\ge \alpha+1$.
In both cases we conclude that $H\not\subseteq G[X]$, so $\rho$ is not potentially $H$-graphic.

It remains to show that $\|\rho- \pi \| > \epsilon n$ for every $\pi \in \mathcal{P}(H,n)$ and some choice of $\epsilon$. 
It suffices to consider the term in position $k-\alpha-1$ in each sequence.
This term is $\CL{\frac{n+k-\alpha-2}{2}}$ in $\rho$.
In $\pi$, this term is either $n-1$ or $k-j+\nabla_j-1$.
In both cases, the difference between these terms is greater than $n/3$, for sufficiently large $n$.
Therefore, $\|\rho-\pi\|>n/3$ for all $\pi\in \mathcal P(H,n)$ when $n$ is sufficiently large.
Thus $H$ is not $\sigma$-stable.
\end{proof}

We are now ready to prove Theorem \ref{thm:MainHigh}.
\begin{proof}  Let $H$ be a graph that satisfies the hypotheses of Theorem \ref{thm:MainHigh}.
First, if $H \not\subseteq K_{k-\alpha-2} \vee S_{b_1,b_2}$ for any $b_1$ and $b_2$ with $b_1 + b_2=\alpha$, then Theorem~\ref{thm:NotStable} implies that $H$ is not $\sigma$-stable.

Now suppose that $H \subseteq K_{k-\alpha-2} \vee S_{b_1,b_2}$ for some $b_1$ and $b_2$ with $b_1 \geq b_2$ and $b_1 + b_2=\alpha$.
Let $\epsilon > 0$ be given, and let $\pi=(d_1, \ldots, d_n)$ be a nonincreasing graphic sequence that is not potentially $H$-graphic satisfying $\sigma(\pi) \geq (2(k-i^*)+\nabla_{i^*}-1-\delta)n$ for some $n\ge n_0$, where $\delta$ and $n_0$ are given by Lemma~\ref{Const}. 
It follows from Lemma \ref{Const} that either $\|\pi-\pi' \| < \epsilon n$ for some $\pi' \in \mathcal{P}(H,n)$, or $\pi$ is potentially $(K_{k-\alpha-1} \vee \overline{K}_{\alpha+1})$-graphic. 
We assume that there is an appropriate choice of $\pi$ that is potentially $(K_{k-\alpha-1} \vee \overline{K}_{\alpha+1})$-graphic, as otherwise $H$ is $\sigma$-stable. 

Let $G$ be a realization of $\pi$ on the vertex set $\{v_1, \ldots, v_n\}$ with $d(v_i) = d_i$ for $i\in\{1,\ldots, n\}$.
Let $Q=\{v_1, \ldots, v_{k-\alpha-1}\}$ and let $R=\{v_{k-\alpha}, \ldots, v_{k}\}$.
By Theorem \ref{yin_complete_split}, we may assume that $G$ contains a copy of $K_{k-\alpha-1} \vee \overline{K}_{\alpha+1}$ on $Q\cup R$ such that $Q$ is the clique of order $k-\alpha-1$ and $R$ is the independent set of order $\alpha+1$. 
If $R$ is not an independent set in $G$, then $G$ contains a copy of $H$ on the vertex set $Q\cup R$, contradicting the assumption that $\pi$ is not potentially $H$-graphic, so assume that the subgraph of $G$ induced by $R$ contains no edges.  

\begin{claim}\label{claim:maxdeg}
$d_{k-\alpha}<2k^2$.
\end{claim}

\begin{proof}[Proof of Claim \ref{claim:maxdeg}]
Suppose to the contrary that $d_{k-\alpha}\ge 2k^2$.
If $v_{k-\alpha}$ has $\alpha$ neighbors in $V(G)-(Q\cup R)$ that are all adjacent to every vertex in $Q$, then $G$ contains $K_{k-\alpha}\vee \overline K_{\alpha}$.
Hence $G$ contains $H$, contradicting the assumption that $\pi$ is not potentially $H$-graphic.
For $j \in \{1, \ldots, k-\alpha-1\}$, let $W_j$ be the set of neighbors of $v_{k-\alpha}$ that are not adjacent to $v_j$. 
Recall that $H \subseteq K_{k-\alpha-2} \vee S_{b_1, b_2}$ for some $b_1$ and $b_2$ with $b_1+b_2=\alpha$.
Since each neighbor of $v_{k-\alpha}$ outside of $Q\cup R$ is in some $W_j$ and $d_{k-\alpha}\ge 2k^2\ge (k-\alpha-1)b_2+(k-\alpha-1)+(\alpha-1)$, the pigeonhole principle implies that there is some $p\in \{1, \ldots, k-\alpha-1\}$ such that $|W_p| \geq b_2$. We will use edge exchanges to create a copy of $S_{b_1,b_2}$ on $R\cup\{v_p\}$ with centers $v_{k-\alpha}$ and $v_p$ such that every vertex in this double star is adjacent to every vertex in $Q-\{v_p\}$.  

Let $\{x_1,\ldots,x_{b_2}\}\subseteq W_p$.
For each $i\in\{1,\ldots,b_2\}$, exchange the edges $x_iv_{k-\alpha}$ and $v_p v_{k-\alpha+i}$ with the nonedges $x_i v_p$ and $v_{k-\alpha} v_{k-\alpha+i}$ (see Figure~\ref{Fig:maxdeg}).
After these edge exchanges, $v_{k-\alpha}$ is adjacent to each of $v_{k-\alpha+1}, \ldots, v_{k-\alpha+b_2}$, and $v_p$ remains adjacent to $v_{k-\alpha}$ and $v_{k-\alpha+b_2+1},\dots, v_{k}$.  
Thus there is a realization of $\pi$ containing $K_{k-\alpha-2} \vee S_{b_1, b_2}$ where $v_p$ and $v_{k-\alpha}$ are the centers of the double star and $Q-\{v_p\}$ is the clique.  This contradicts the assumption that $\pi$ is not potentially $H$-graphic.

\begin{figure}
\begin{tikzpicture}
\node at (0,-1) {$R$};
\draw (0,2.5) ellipse (2 and .5);
\node at (0,3.5) {$Q$};
\draw (0,0) ellipse (2.5 and .75);

\draw (-2.5,0)--(-2,2.5);
\draw (2.5,0)--(2,2.5);

\fill (1,2.5) \myvert;
\node at (.8,2.7) {$v_p$};
\fill (1.5,0) \myvert;
\node at (1.5,-.3) {$v_{k-\alpha}$};

\draw (0,0) ellipse (.75 and .4);

\draw [dashed] (1.5,0)--(.3,.3666);
\draw [dashed] (1.5,0)--(.3,-.3666);

\draw (1,2.5)--(.75,0);
\draw (1,2.5)--(-.7,.1436);

\draw (3.5,.5) ellipse (.4 and .75);

\draw (1.5,0)--(3.5,-.25);
\draw (1.5,0)--(3.4,1.2261);

\draw [dashed] (1,2.5)--(3.5,1.25);
\draw [dashed] (1,2.5)--(3.3,-.1495);

\node at (3.5,.5) {$W_p$};

\fill (.5,0) \myvert;
\fill (-.5,0) \myvert;
\node at (0,0) {$\ldots$};
\end{tikzpicture}
\caption{Edge exchanges for Claim~\ref{claim:maxdeg}.}\label{Fig:maxdeg}

\end{figure}
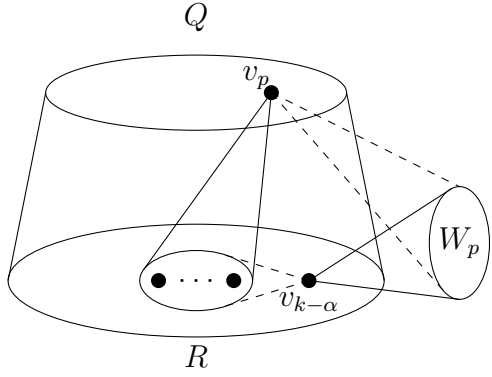
\end{proof}

We note that with some care in maximizing $$(k-\alpha-1)b_2+(k-\alpha-1)+(\alpha-1),$$ the assumption that $d_{k-\alpha}<2k^2$ could be strengthened considerably.  However, this is not needed to obtain our result, and the given bound makes some forthcoming calculations simpler.

\begin{claim}\label{claim:fewbigdeg}
$d_{k-\alpha+8k^4}\le k-\alpha-1$.
\end{claim}

\begin{proof}[Proof of Claim \ref{claim:fewbigdeg}]
Suppose to the contrary that $d_{k-\alpha+8k^4}> k-\alpha-1$.
Let $u,v\in R$ and note that $d(u)\ge k-\alpha$ and $d(v)\ge k-\alpha$.
Since $R$ is an independent set in $G$, there are vertices $a_1$ and $a_2$ in $V(G)-(Q\cup R)$ such that $ua_1\in E(G)$ and $va_2\in E(G)$ (it is possible that $a_1=a_2$). 
Consider the graph $G-Q$.
By Claim~\ref{claim:maxdeg}, $\Delta(G-Q)<2k^2$, and therefore there are at most $8k^4$ vertices that are distance at most $2$ from $a_1$ or $a_2$ in $G-Q$.
By assumption, there are at least $8k^4+1$ vertices of positive degree in $G-Q$, so there is a vertex $w$ with positive degree that is distance at least $3$ from both $a_1$ and $a_2$.
Letting $x$ be a neighbor of $w$, we can exchange the edges $ua_1$, $va_2$, and $wx$ for the nonedges $uv$, $w a_1$, and $xa_2$ so that $R$ induces exactly one edge (see Figure~\ref{Fig:fewbigdeg}).
This realization contains $H$, contradicting the assumption that $\pi$ is not potentially $H$-graphic.
\end{proof}

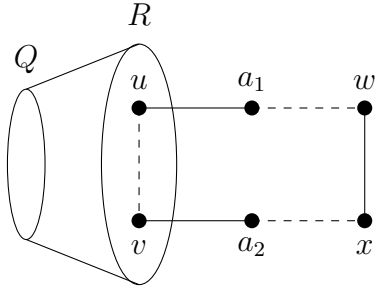
\begin{figure}
\begin{tikzpicture}
\node at (0,1.4) {$Q$};
\draw (0,0) ellipse (.25 and 1);
\node at (1.5,2) {$R$};
\draw (1.5,0) ellipse (.5 and 1.6);

\draw (0,1)--(1.5,1.6);
\draw (0,-1)--(1.5,-1.6);

\fill (1.5,.75) \myvert;
\node at (1.5,1.1) {$u$};
\fill (1.5,-.75) \myvert;
\node at (1.5,-1.1) {$v$};

\fill (3,.75) \myvert;
\node at (3,1.1) {$a_1$};
\fill (3,-.75) \myvert;
\node at (3,-1.1) {$a_2$};

\fill (4.5,.75) \myvert;
\node at (4.5,1.1) {$w$};
\fill (4.5,-.75) \myvert;
\node at (4.5,-1.1) {$x$};

\draw (1.5,.75)--(3,.75);
\draw (1.5,-.75)--(3,-.75);
\draw (4.5,.75)--(4.5,-.75);

\draw [dashed] (1.5,.75)--(1.5,-.75);
\draw [dashed] (3,.75)--(4.5,.75);
\draw [dashed] (3,-.75)--(4.5,-.75);

\end{tikzpicture}
\caption{Edge exchanges for Claim~\ref{claim:fewbigdeg}.}\label{Fig:fewbigdeg}
\end{figure}

\begin{claim}
$\|\pi-\wtilpi_{k-\alpha-1}\|<\epsilon n$.
\end{claim}

\begin{proof}[Proof of Claim \ref{claim:fewbigdeg}]
By Claim~\ref{claim:fewbigdeg}, we know that $d_{k-\alpha+8k^4}\le k-\alpha-1$.
The only terms in $\pi$ that may exceed their corresponding terms in $\wtilpi_{\alpha+1}(H,n)$ are $d_j$ for $j\in\{k-\alpha,\ldots,k-\alpha+8k^4-1\}$.
By Claim~\ref{claim:maxdeg}, these terms in $\pi$ are bounded above by $2k^2$, so the sum of the absolute differences of these terms in $\pi$ and $\wtilpi_{\alpha+1}(H,n)$ is at most $16k^6$.
Since $\|\pi-\wtilpi_{\alpha+1}(H,n)\|$ is bounded above by $|\sigma(\wtilpi_{\alpha+1}(H,n))-\sigma(\pi)|$ plus twice the absolute difference of the terms where $\pi$ exceeds $\wtilpi_{\alpha+1}(H,n)$, it follows that
\begin{align*}
\|\pi-\wtilpi_{\alpha+1}(H,n)\|&\le|\sigma(\wtilpi_{\alpha+1}(H,n)-\sigma(\pi)|+16k^6\\
&\le \delta n+16k^6.
\end{align*}
For $n$ sufficiently large, $16k^2<\frac \epsilon 2n$.
Since $\delta<\epsilon/2$ it follows that if $\pi$ is not potentially $H$-graphic, then $\| \pi -\wtilpi_{\alpha+1}(H,n)\|<\epsilon n$.
\end{proof}

Since $\wtilpi_{\alpha+1}(H,n)\in\mathcal P(H,n)$, we conclude that $H$ is $\sigma$-stable. This concludes the proof of Theorem \ref{thm:MainHigh}.
\end{proof}

\section{Weak $\sigma$-stability}\label{sec:notstable}\label{sec:weak}

We conclude by briefly introducing a weaker version of stability that is distinct from that we have discussed throughout this paper.
As an example in the introduction, we showed that $K_3$ is not $\sigma$-stable using the degree sequence $\pi=\rho(K_3,n)=((\frac{n}{2})^2,1^{n-2})$, which is actually an instance of Theorem~\ref{thm:NotStable}.
In some sense, this example is unsatisfying since $\pi$ fails an obvious necessary condition for being potentially $K_3$-graphic: it is not degree-sufficient for $K_3$.
Indeed, for all $k\ge 3$, the sequence $\rho(K_k,n)$ defined in the proof of Theorem~\ref{thm:NotStable} fails to be degree sufficient for $K_k$.  This motivates the following notion of $\sigma$-stability.

\begin{defn} A graph $H$ is {\bf weakly $\sigma$-stable} if for any $\epsilon> 0$, there exists an $n_0=n(\epsilon, H)$ and $\delta > 0$ such that for any graphic sequence $\pi$ of length $n \geq n_0$ that is degree-sufficient for $H$ but not potentially $H$-graphic and satisfies
$$\sigma (\pi) \geq \sigma(H,n)-\delta n,$$
there is some $\pi' \in \mathcal{P}(H,n)$ such that $\|\pi-\pi'\| < \epsilon n$.
\end{defn}

Note that if $H$ satisfies the conditions of Theorem \ref{thm:NotStable} and, in addition, the sequence $\rho(H,n)$ is degree-sufficient for $H$, then $H$ is not weakly $\sigma$-stable.  For example, $C_6$, the cycle on 6 vertices, is not weakly $\sigma$-stable, because it is Type 2, but is not a subgraph of $K_1 \vee S_{(n-1)/2,(n-1)/2}$, and the sequence $\rho(C_6,n)$ is degree-sufficient for $C_6$.  On the other hand, we can show that complete graphs are weakly $\sigma$-stable even though they are not $\sigma$-stable.  

\begin{thm}\label{CliqueWeak}
For all $k \geq 3$, the complete graph $K_k$ is weakly $\sigma$-stable.
\end{thm}

\begin{proof}
If a nonincreasing sequence $\pi=(d_1, \ldots, d_n)$ is degree-sufficient for $K_k$, then $d_k \geq k-1$.
By Theorem \ref{thm:YinLi}, if $\pi$ is not potentially $K_k$-graphic, then $d_{2k} < k-3$.
Thus $d_j < k-3$ for each $j \geq 2k$, as well there being some $i\le k-2$ such that $d_i < 2(k-1)-i$. 
Hence if $\pi$ is not potentially $H$-graphic, then it is termwise bounded above by the sequence $((n-1)^{k-3}, (k-1)^{k+2}, (k-3)^{n-2k+1})$, so $\sigma(\pi)\le 2(k-3)n-k^2+7k-2$. 
The potential number for $K_k$ is $\sigma(K_k, n) = 2(k-2)n-k^2+3k$ (see \cite{GJL, LiSong, LiSong2, LiSongLuo}).
Thus, $\sigma(K_k, n)-\sigma(\pi)\ge 2n-4k+2$. 
Therefore among graphic sequences that are degree-sufficient for $K_k$ but not potentially $K_k$-graphic, there are none that satisfy the condition that $\sigma(\pi) \geq \sigma(K_k, n)-\delta n$ for $\delta < 2$.
Thus, $K_k$ satisfies the conditions for weak $\sigma$-stability for all $\epsilon>0$ with $n_0\ge 2k$ and $\delta=1$.
\end{proof}

Theorem \ref{CliqueWeak}, together with the above observation about $C_6$ together demonstrate that the class of $\sigma$-stable and weakly $\sigma$-stable graphs are not identical.  This suggests a potentially fruitful line of inquiry going forward.

\end{document}